\documentclass[11pt]{article}

\usepackage[top=1in, bottom=1.25in, left=1.25in, right=1.25in]{geometry}
\usepackage{amsmath}
\usepackage{amsfonts}
\usepackage{amsthm}
\usepackage{amssymb}
\usepackage{thm-restate}
\usepackage{enumerate}
\usepackage{framed}
\usepackage{mathtools}
\usepackage{xifthen}
\usepackage{float}
\usepackage{subcaption}
\usepackage{tikz}
\usepackage{amsrefs}

\usetikzlibrary{patterns}

\usepackage{hyperref}
\usepackage{cleveref}

\newtheorem{thm}{Theorem}
\crefname{thm}{Theorem}{Theorems}
\Crefname{thm}{Theorem}{Theorems}

\newtheorem{lem}[thm]{Lemma}
\crefname{lem}{Lemma}{Lemmas}
\Crefname{lem}{Lemma}{Lemmas}

\newtheorem{prop}[thm]{Proposition}
\crefname{prop}{Proposition}{Propositions}
\Crefname{prop}{Proposition}{Propositions}

\newtheorem{cor}[thm]{Corollary}
\crefname{cor}{Corollary}{Corollaries}
\Crefname{cor}{Corollary}{Corollaries}

\newtheorem{conj}[thm]{Conjecture}
\crefname{conj}{Conjecture}{Conjectures}
\Crefname{conj}{Conjecture}{Conjectures}

\crefname{qn}{Question}{Questions}
\Crefname{qn}{Question}{Questions}

\theoremstyle{definition}

\DeclareMathAlphabet\mathbfcal{OMS}{cmsy}{b}{n}

\newenvironment{lemprimed}[1]%
{%
    \renewcommand{\thelem}{\protect\ref{#1}'\protect}%
    \addtocounter{thm}{-1}
    \begin{lem}%
}{%
    \end{lem}%
}

\newenvironment{thmprimed}[1]%
{%
    \renewcommand{\thethm}{\protect\ref{#1}'\protect}%
    \addtocounter{thm}{-1}
    \begin{thm}%
}{%
    \end{thm}%
}

\newcommand{\Z}{\mathbb{Z}}
\newcommand{\Q}{\mathbb{Q}}
\newcommand{\R}{\mathbb{R}}

\newcommand{\F}{\mathcal{F}}

\newcommand{\es}{\emptyset}
\newcommand{\sm}{\setminus}

\newcommand{\BL}[1]{2^{[#1]}}

\newcommand{\ind}[1]{1_{#1}}
\newcommand{\sym}[1]{{#1}^{\textrm{sym}}}
\newcommand{\widebar}[1]{\bar{#1}}

\newcommand{\Ac}{\widebar{A}}
\newcommand{\Bc}{\widebar{B}}
\newcommand{\rpart}[1]{$#1$-partition}
\newcommand{\modpart}[1]{$(1\text{ mod }#1)$-partition}

\newcommand{\clone}{copy}

\newcommand{\posetcopy}{copy}

\begin{document}

\title{Partitioning the Boolean lattice into copies of a poset}

\author{%
    Vytautas Gruslys\thanks{%
        Department of Pure Mathematics and Mathematical Statistics,
        University of Cambridge,
        Wilberforce Road,
        CB3 0WA Cambridge,
        United Kingdom;
        e-mail: \{\texttt{v.gruslys,i.leader,i.tomon}\}\texttt{@cam.ac.uk}\,.
    }
    \and
    Imre Leader\footnotemark[1]
    \and
    Istv\'an Tomon\footnotemark[1]
}

\maketitle

\begin{abstract}

    Let $P$ be a poset of size $2^k$ that has a greatest and a least element. We
    prove that, for sufficiently large $n$, the Boolean lattice $\BL{n}$ can be
    partitioned into copies of $P$.  This resolves a conjecture of Lonc.

\end{abstract}

\section{Introduction} \label{sec:introduction}

    Let $\BL{n}$ denote the \emph{Boolean lattice} of dimension $n$, that is,
    the poset (partially ordered set) whose elements are the subsets of $[n] =
    \{1, \dotsc, n\}$, ordered by inclusion.

    An important property of the Boolean lattice is that any finite poset $P$
    can be embedded into $\BL{n}$ for sufficiently large $n$. Here by an
    \emph{embedding} of a poset $P$ into a poset $Q$ we mean an injection $f : P
    \to Q$ such that $f(x) \le_Q f(y)$ if and only if $x \le_P y$. For any
    embedding $f : P \to Q$, we call the image $f(P)$ a \emph{copy} of $P$ in
    $Q$.

    Now, if $P$ is fixed and $n$ is large, then $\BL{n}$ contains many copies of
    $P$.  So a natural question arises: can $\BL{n}$ be partitioned into copies
    of $P$?  Of course, for such a partition to exist, the size of $P$ must
    divide the size of $\BL{n}$, that is, $|P|$ must be a power of $2$ (we would
    like to emphasise that we denote by $|P|$ the number of elements of $P$ and
    not the number of relations).  Moreover, $P$ must have a greatest and a
    least element. Lonc~\cite{Lonc1991} conjectured that these obvious necessary
    conditions are in fact sufficient.

    \begin{conj}[Lonc] \label{conj:main}

        Let $P$ be a poset of size $2^k$ with a greatest and a least element.
        Then, for sufficiently large $n$, the Boolean lattice $\BL{n}$ can be
        partitioned into copies of $P$.

    \end{conj}

    The case where $P$ is a chain of size $2^k$ was originally conjecture by
    Sands~\cite{Sands1985}. Griggs~\cite{Griggs1988} proposed a slightly
    stronger conjecture that, for any positive integer $c$ and for sufficiently
    large $n$, it is possible to partition $\BL{n}$ into chains of length $c$
    and at most one other chain. Both conjectures were proved by
    Lonc~\cite{Lonc1991}. The question of minimising the dimension $n$ in
    Griggs' conjecture in terms of the length of the chain $c$ has received
    attention from several authors, including Elzobi and Lonc~\cite{Elzobi2003}
    and Griggs, Yeh and Grinstead~\cite{Griggs1987}. Recently,
    Tomon~\cite{Tomon2016} proved that the smallest sufficient $n$ is of order
    $\Theta(c^2)$. Related questions on partitioning $\BL{n}$ into chains of
    almost equal lengths have also been examined, by F\"uredi \cite{Furedi1985},
    Hsu, Logan, Shahriari and Towse \cite{Hsu2002,Hsu2003} and Tomon
    \cite{Tomon2015}.
    
    As we mentioned in the previous paragraph, Lonc himself verified
    \Cref{conj:main} in the case where $P$ is a chain.  Furthermore, it is easy
    to extend this result to products of chains. In fact, for any two posets $P,
    Q$, if $\BL{n}$ can be partitioned into copies of $P$ and $\BL{m}$ can be
    partitioned into copies of $Q$, then $\BL{n+m}$ can be partitioned into
    copies of $P \times Q$.  However, apart from some small cases that can be
    checked by hand, chains and their products were the only two cases for which
    Lonc's conjecture had been confirmed.

    In this paper we resolve the conjecture in full generality.

    \begin{thm} \label{thm:main}

        Let $P$ be a poset of size $2^k$ with a greatest and a least
        element. Then, for sufficiently large $n$, the Boolean lattice $\BL{n}$
        can be partitioned into copies of $P$.

    \end{thm}

    The plan of the paper is as follows.  In \Cref{sec:overview} we give the
    most important definitions and outline the structure of the proof of
    \Cref{thm:main}. We give the actual proof in
    \Cref{sec:general,sec:particular}: \Cref{sec:general} contains a general
    argument, which works in various settings where a partition of a product set
    into smaller sets is sought, and might be of independent interest; 
    \Cref{sec:particular} contains ideas that are particular to
    partitioning $\BL{n}$ into copies of a fixed poset. Finally, in
    \Cref{sec:final} we give some open problems.

\section{Overview of the proof} \label{sec:overview}

\subsection{Weak partitions}

A key idea in the proof will be the interplay between partitions and two weaker
notions, called \rpart{r}s and \modpart{r}s, which we now describe.  This idea
appears (in a different context) in a paper by Gruslys, Leader and
Tan~\cite{Gruslys2016}.

Let $P$ be a poset. Recall that a set $A \subset \BL{n}$ is a \emph{\posetcopy}
of $P$ if the poset induced on $A$ by $\BL{n}$ is isomorphic to $P$.  We define
$\mathcal{F}_n(P)$ to be the family of all copies of $P$ in $\BL{n}$. 

Let $X$ be a set, and let $\mathcal{F}$ be any family of subsets of $X$. A
\emph{weight function} on $\mathcal{F}$ is an assignment of non-negative
integer weights to the members of $\mathcal{F}$. For an element $x \in X$, the
\emph{multiplicity} of $x$ for a weight function is the total weight of those
members of $\mathcal{F}$ that contain $x$. So, for example, $X$ can be
partitioned into members of $\mathcal{F}$ if and only if there exists a weight
function on $\mathcal{F}$ for which every element of $X$ has multiplicity $1$.
For a positive integer $r$, we say that

\begin{itemize}
    \item $\mathcal{F}$ \emph{contains an $r$-partition of} $X$ if there is a
        weight function on $\mathcal{F}$ for which every element of $X$ has
        multiplicity $r$\,;
    \item $\mathcal{F}$ \emph{contains a $(1\text{ mod }r)$-partition of} $X$ if
        there is a weight function on $\mathcal{F}$ for which every $x \in X$
        has multiplicity $1 + rk_x$, where $k_x \in \{0,1,\dotsc\}$ may depend on
        $x$.
\end{itemize}

Our strategy revolves around establishing a close relation between \rpart{r}s,
\modpart{r}s and actual partitions of sets. Obviously, if $\F$ contains a
partition of $X$, then $\F$ contains an \rpart{r} and a \modpart{r} of $X$ for
every $r$. Our aim is to go in the opposite direction.  Namely, our strategy
consists of two steps: firstly, we will show that if there exists an $r$ such
that $\F$ contains an \rpart{r} and a \modpart{r} of $X$, then we can use these
weak partitions to get an actual partition of $X^m$ for some $m$; secondly, we
will show that, for some $n$ and $r$, $\F_n(P)$ does contain an \rpart{r} and a
\modpart{r} of $\BL{n}$.

It is not immediately obvious that this strategy should work. For instance, it
is not clear that finding weak partitions of $\BL{n}$ is easier than finding an
actual partition. However, this will turn out to be the case in
\Cref{sec:particular}, where we prove the following lemmas.

\begin{restatable}{lem}{rpartition} \label{lem:r-partition}

    Let $P$ be a finite poset with a greatest and a least element. Then there
    exist positive integers $n$ and $r$ such that the family of copies of $P$ in
    $\BL{n}$ contains an $r$-partition of $\BL{n}$.

\end{restatable}

\begin{restatable}{lem}{modrpartition} \label{lem:1-mod-r-partition}

    Let $P$ be a finite poset of size $2^k$ that has a greatest and a least
    element, and let $r$ be a positive integer.  Then there exists a positive
    integer $n$ such that the family of copies of $P$ in $\BL{n}$ contains a
    $(1\text{ mod }r)$-partition of $\BL{n}$.

\end{restatable}

A key part of the argument will be to see how to use these seemingly much weaker
results can be used to find an actual partition of $\BL{n}$. We will discuss
this in the following subsection.

\subsection{Product systems}

We will prove a very general theorem, which, applied to \Cref{lem:r-partition,%
lem:1-mod-r-partition}, will imply our main result. 

Let $S$ be a set. For two sets $A \subset S^m, B \subset S^n$ with $m \le n$, we
say that $B$ is a \emph{\clone{}} of $A$ if $B$ can be obtained by taking a
product of $A$ with a singleton set in $S^{n-m}$ and permuting the coordinates.
More precisely, for a permutation $\pi$ of $\{1, \dotsc, n\}$ and $x = (x_1,
\dotsc, x_n) \in S^n$, we define $\pi(x) = (x_{\pi(1)}, \dotsc, x_{\pi(n)})$.
Moreover, for any $X \subset S^n$, we define $\pi(X) = \{\pi(x) : x \in X\}$.
Finally, for any $X \subset S^m$ and $Y \subset S^{n-m}$, we define $X \times Y
= \{ (x_1, \dotsc, x_m, y_1, \dotsc, y_{n-m}) : (x_1, \dotsc, x_m) \in X, (y_1,
\dotsc, y_{n-m}) \in Y \}$. Note that we abuse the notation slightly and
identify $S^m \times S^{n-m}$ with $S^n$, which allows us to consider $X \times
Y$ as a subset of $S^n$. With these definitions, $B$ is a copy of $A$ if $B =
\pi(A \times \{y\})$ for some permutation $\pi$ of $\{1, \dotsc, n\}$ and some $y
\in S^{n-m}$.

Note that this definition does not exactly agree with the definition of a copy
of a poset, which we made in \Cref{sec:introduction}. Indeed, there may exist
two sets $A, B \subset \BL{n}$ such that $\BL{n}$ induces the same poset on $A$
and $B$, but such that $B$ cannot be obtained from $A$ by permuting the
coordinates. However, we think that this abuse of notation is not harmful,
because it will always be clear from the context which definition of a copy
should be used. Moreover, if sets $A \subset \BL{n}$ and $B \subset \BL{m}$ are
copies in the new sense, then they are also copies when considered as posets.
Therefore, the two definitions are in fact closely related.

The following theorem is vital for our strategy.

\begin{restatable}{thm}{general} \label{thm:general}

    Let $S$ be a finite set and let $\F$ be a family of subsets of $S$. Suppose
    that there exists a positive integer $r$ such that $\F$ contains an
    \rpart{r} and a \modpart{r} of $S$. Then there exists a positive integer $n$
    such that $S^n$ can be partitioned into copies of members of $\mathcal{F}$.

\end{restatable}

This theorem was inspired by work of Gruslys, Leader and Tan~\cite{Gruslys2016}.
They implicitly used a special case of this theorem to prove that, for any
finite (non-empty) set $T \subset \mathbb{Z}^k$, there exists a positive integer
$n$ such that $\mathbb{Z}^n$ can be partitioned into isometric copies of $T$.

It is straightforward to deduce our main theorem from \Cref{lem:r-partition,%
lem:1-mod-r-partition,thm:general}. Indeed, let $P$ be a poset of size $2^k$
with a greatest and a least element. \Cref{lem:r-partition} implies that there
are positive integers $r$ and $u$ such that $\mathcal{F}_{u}(P)$ contains an
\rpart{r} of $\BL{u}$.  Now \Cref{lem:1-mod-r-partition} implies that there is a
positive integer $v$ such that $\mathcal{F}_{v}(P)$ contains a \modpart{r} of
$\BL{v}$. Setting $m = \max\{u, v\}$, $\mathcal{F}_m(P)$ contains both an
\rpart{r} and a \modpart{r} of $\BL{m}$.  We can now apply \Cref{thm:general}
with $\F = \F_m(P)$ and $S = \BL{m}$ to finish the proof.  (Note that if $B
\subset \BL{mn}$ is a \clone{} of some $A \in \F_m(P)$, then the poset that
$\BL{mn}$ induces on $B$ is isomorphic to $P$, and hence $B \in \F_{mn}(P)$.)

\section{Partitions in product systems} \label{sec:general}

Our aim in this section is to prove \Cref{thm:general}.

\general*

As in the statement of the theorem, we let $\F$ be a family of subsets of a
finite set $S$ and we suppose that $r$ is a natural number such that $\F$
contains an \rpart{r} and \modpart{r} of $S$. The set $S$, family $\F$ and
number $r$ will remain fixed throughout this section.

\begin{restatable}{lem}{main} \label{lem:main}

    For any sets $A, B \subset S$, there exists a positive integer $n$ such that
    $S^2 \times (A \cup B)^n$ can be partitioned into copies of members of
    $\mathcal{F} \cup \{A, B\}$.

\end{restatable}

The proof of \Cref{lem:main} is by far the most complicated part of this paper.
We will prove \Cref{lem:main} in the next subsection. Now, with \Cref{lem:main}
at our disposal, we will prove \Cref{thm:general}.

\begin{prop} \label{prop:buildbigger}

    Let $A, B \subset S$ and suppose that there exist positive integers $p, q$
    such that
    \begin{itemize}
        \item $S^p$ can be partitioned into copies of members of $\F \cup
            \{A\}$, and
        \item $S^2 \times A^q$ can be partitioned into copies of members of $\F
            \cup \{B\}$.
    \end{itemize}
    Then $S^{pq+2}$ can be partitioned into copies of members of $\F \cup
    \{B\}$.

\end{prop}

\begin{proof}
    Partition $S^p$ into sets $X_1, \dotsc, X_u, Y_1, \dotsc, Y_v$, where every
    $X_i$ is a copy of $A$ and every $Y_j$ is a copy of a member of $\F$.  We
    denote $\mathcal{X} = \{X_1, \dotsc, X_u\}$ and $\mathcal{Y} = \{Y_1,
    \dotsc, Y_v\}$. Then $S^{pq+2} = S^2 \times (S^p)^q$ is the disjoint union of
    sets $S^2 \times Z_1 \times \dotsb \times Z_q$ with $Z_i \in \mathcal{X} \cup
    \mathcal{Y}$ for all $i$. We separate these sets into two families, namely,
    \begin{align*}
        \mathcal{A} &= \big\{ S^2 \times Z_1 \times \dotsb \times Z_q : Z_i \in
            \mathcal{X} \text{ for all } i \big\}, \\
        \mathcal{B} &= \big\{ S^2 \times Z_1 \times \dotsb \times Z_q : Z_i \in
            \mathcal{X} \cup \mathcal{Y} \text{ for all } i
            \text{ and } Z_j \in \mathcal{Y} \text{ for some } j \big\}.
    \end{align*}
    Each member of $\mathcal{A}$ is a copy of $S^2 \times A^q$, so it can be
    partitioned into copies of members of $\F \cup \{B\}$. Moreover, each member
    of $\mathcal{B}$ can be partitioned into copies of some member of $\F$ in an
    obvious way. Since together these sets form a partition of $S^{pq+2}$, we
    are done.
\end{proof}

\begin{proof}[Proof of \Cref{thm:general} (assuming \Cref{lem:main})]
    Since $\mathcal{F}$ contains an $r$-partition of $S$ with $r \ge 1$, and
    since $S$ is finite, we can find finitely many sets $B_1, \dotsc, B_k \in
    \F$ that cover $S$. We define $A_i = B_1 \cup \dotsb \cup B_i$ for every $1
    \le i \le k$.  So, in particular, $A_k = S$.

    We will use reverse induction on $i$ to prove that there exist positive
    integers $p_1, \dotsc, p_k$ such that, for every $1 \le i \le k$, $S^{p_i}$
    can be partitioned into copies of members of $\F \cup \{A_i\}$. If $i = k$,
    then $A_k = S$, and the statement is trivially true with, say, $p_k = 1$.
    So we may assume that $1 \le i \le k-1$. Since $A_{i+1} = A_i \cup B_{i+1}$,
    it follows from \Cref{lem:main} that there exists a positive integer $q$
    such that $S \times (A_{i+1})^q$ can be partitioned into copies of members
    of $\F \cup \{A_i, B_{i+1}\}$. However, $B_{i+1}$ is a member of $\F$, so
    $\F \cup \{A_i, B_{i+1}\} = \F \cup \{A_i\}$. Combining this with the
    induction hypothesis for $i+1$ and \Cref{prop:buildbigger}, we see that
    $S^{p_i}$, where $p_i = p_{i+1}q+2$, can be partitioned into copies of
    members of $\F \cup \{A_i\}$.

    In particular, the statement holds for $i = 1$. Since $A_1 = B_1 \in \F$, it
    says that $S^{p_1}$ can be partitioned into copies of members of $\F$, as
    required.
\end{proof}

\subsection{Proof of \Cref{lem:main}}

Here we will prove \Cref{lem:main}.

\main*

We start by picking two sets $A, B \subset S$; these sets will be fixed
throughout the subsection. We define $U = A \cup B$, $\Ac = U \sm A$ and $\Bc =
U \sm B$. Moreover, for any integers $1 \le i \le d$, we define

$$
    C_{i,d} =%
        \underbrace{%
            \Ac \times \dotsb \times \Ac \times
            \overset{%
                \overset{%
                    \mathclap{i\text{-th component}}
                }{%
                    \downarrow
                }
            }{%
                \Bc 
            }
                \times \Ac \times \dotsb \times \Ac
        }_{d \text{ components}}.
$$
We also define $C_{0,d} = \Ac^d$. Our aim is to prove that there exists a
positive integer $n$ such that $S \times U^n$ can be partitioned into copies
of members of $\F \cup \{A, B\}$.

At certain points in the proof we will be conjuring up extra elbow space by
`blowing up' $S^k$, for some $k$, into $S^{k+1}$. It turns out that sometimes a
set $X \subset S^k$ can be usefully identified with a larger set $X \times \Ac
\subset S^{k+1}$. The following simple proposition is an example of this idea.

\begin{prop} \label{prop:blowup}
    
    Let $k \ge 1$ and let $X \subset U^k$ be such that $U^k \sm X$ can be
    partitioned into copies of $A$ and $B$. Then $U^{k+1} \sm \left( X \times
    \Ac \right)$ can be partitioned into copies of $A$ and $B$.

\end{prop}

\begin{proof}
    Partition $U^{k+1} \sm \left( X \times \Ac \right)$ into sets $\left( U^k
    \sm X \right) \times \Ac$ and $U^k \times A$; the first of these sets can be
    partitioned into copies of $U^k \sm X$, and the second -- into copies of
    $A$.
\end{proof}

If we could prove that $U^k$, for some $k$, can be partitioned into copies of
$A$ and $B$ (that is, without using $\F$), then we would be done. Of course,
this is not possible in general. However, we can partition $U^k$ with one
$C_{i,k}$ removed.

\begin{prop} \label{prop:onecorner}

    For any integers $k \ge 1$ and $0 \le i \le k$, the set $U^k \sm C_{i,k}$
    can be partitioned into copies of $A$ and $B$.

\end{prop}

\begin{proof}
    We use induction on $k$. If $k = 1$, then, depending on the value of $i$,
    $U \sm C_{i,1}$ is either $A$ or $B$. If $k \ge 2$, we may assume that $i
    \neq k$ (in fact, there are only two distinct cases: $i=0$ and $i \neq 0$).
    By the induction hypothesis, $U^{k-1} \sm C_{i,k-1}$ can be partitioned into
    copies of $A$ and $B$. However, $C_{i,k} = C_{i,k-1} \times \Ac$, so we are
    done by \Cref{prop:blowup}.
\end{proof}

\Cref{prop:blowup} says that if we can partition a subset of $U^k$, then we can
also partition an `equivalent' subset of $U^{k+1}$. The following proposition
allows us to use the extra space in $U^{k+1}$ to slightly modify this subset.

\begin{prop} \label{prop:modify}

    Let $X \subset U^k$ be such that $U^k \sm X$ can be partitioned into copies
    of $A$ and $B$. Suppose that $X$ contains the set $C_{i,k}$ for some $0 \le
    i \le k$. Then the set $U^{k+1} \setminus Y$, where
    $$
        Y = \left( X \times \Ac \right) \cup C_{k+1,k+1} \sm C_{i,k+1},
    $$
    can also be partitioned into copies of $A$ and $B$.

\end{prop}

\begin{proof}
    Partition $U^{k+1} \setminus Y$ into four sets $Z_1, Z_2, Z_3, Z_4$, where
    \begin{align*}
        Z_1 &= (U^k \sm C_{0,k}) \times \Bc, \\
        Z_2 &= (U^k \sm C_{i,k}) \times (A \cap B), \\
        Z_3 &= C_{i,k} \times B, \\
        Z_4 &= (U^k \sm X) \times \Ac.
    \end{align*}
    It is evident from \Cref{fig:modify} that these four sets do partition
    $U^{k+1} \setminus Y$.
    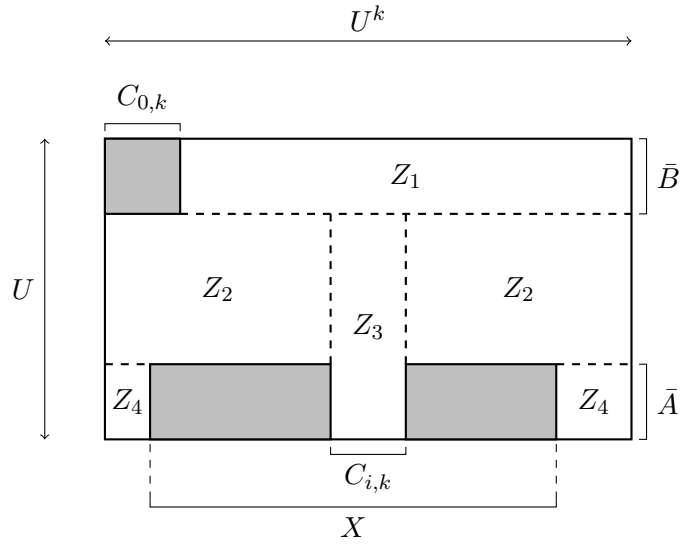
\begin{figure}[ht]
        \begin{centering}
            \begin{tikzpicture}
                \draw[thick] (0,0) -- (7,0) -- (7,4) -- (0,4) -- cycle;
                \draw[thick,dashed] (0,1) -- (0.6,1);
                \draw[thick,dashed] (6,1) -- (7,1);
                \draw[thick,dashed] (1,3) -- (7,3);
                \draw[thick,dashed] (3,1) -- (3,3);
                \draw[thick,dashed] (4,1) -- (4,3);
                \draw[thick,fill=lightgray] (0.6,0) -- (3,0) -- (3,1) -- (0.6,1) -- cycle;
                \draw[thick,fill=lightgray] (4,0) -- (6,0) -- (6,1) -- (4,1) -- cycle;
                \draw[thick,fill=lightgray] (0,3) -- (1,3) -- (1,4) -- (0,4) -- cycle;
                \draw (4, 3.5) node {$Z_1$};
                \draw (1.5, 2) node {$Z_2$};
                \draw (5.5, 2) node {$Z_2$};
                \draw (3.5, 1.5) node {$Z_3$};
                \draw (0.3, 0.5) node {$Z_4$};
                \draw (6.5, 0.5) node {$Z_4$};
                \draw (0, 4.1) -- (0,4.2) -- node[above]{$C_{0,k}$} (1,4.2) -- (1,4.1);
                \draw (3, -0.1) -- (3, -0.2) -- node[below]{$C_{i,k}$} (4, -0.2) -- (4, -0.1);
                \draw (0.6, -0.8) -- (0.6, -0.9) -- node[below]{$X$} (6, -0.9) -- (6, -0.8);
                \draw[dashed] (0.6, -0.8) -- (0.6, 0);
                \draw[dashed] (6, -0.8) -- (6, 0);
                \draw (7.1,0) -- (7.2,0) -- node[right]{$\Ac$}(7.2, 1) -- (7.1, 1);
                \draw (7.1,3) -- (7.2,3) -- node[right]{$\Bc$}(7.2, 4) -- (7.1, 4);
                \draw[<->] (0,5.3) -- node[above]{$U^k$} (7,5.3);
                \draw[<->] (-0.8,0) -- node[left] {$U$} (-0.8,4);
            \end{tikzpicture}
            \caption{The set $Y \subset U^{k+1}$ is shaded. The four sets $Z_1,
            Z_2, Z_3, Z_4$ partition $U^{k+1} \sm Y$.}
            \label{fig:modify}
        \end{centering}
    \end{figure}
    The sets $Z_1$ and $Z_2$ can be partitioned into copies of $A$ and $B$ by
    \Cref{prop:onecorner}. The set $Z_3$ is obviously a union of disjoint copies
    of $B$. Finally, $Z_4$ is a union of disjoint copies of $U^k \setminus X$,
    so it can be partitioned into copies of $A$ and $B$ by the assumption on
    $X$.
\end{proof}

The previous proposition enables us to make one change to the set $X$ when we go
one dimension up, that is, from $U^k$ to $U^{k+1}$. To make multiple changes, we
apply this proposition multiple times. This is exactly the content of
\Cref{cor:multiplechanges}.

\begin{cor} \label{cor:multiplechanges}
    
    Let $k,l$ be non-negative integers and let $I \subset \{0, \dotsc, k\}, J
    \subset \{k+1, \dotsc, k+l\}$ be sets such that $|J| = |I|$. Then the set
    $U^{k+l} \sm Y$, where
    $$
        Y = \left( U^k \times \Ac^l \right)
          \cup \left( \bigcup_{j \in J} C_{j,k+l} \right)
          \sm \left( \bigcup_{i \in I} C_{i,k+l} \right),
    $$
    can be partitioned into copies of $A$ and $B$.
\end{cor}

\begin{proof}
    We shall apply induction on $l$. If $l=0$, then $|J|=|I|=0$, so $U^k \sm Y =
    \es$, and hence the conclusion trivially holds.

    Now suppose that $l \ge 1$. We will split the argument into two cases,
    depending on whether or not $k+l \in J$.  If $k+l \in J$, then we write
    $j^\ast = k+l$ and we pick any $i^\ast \in I$.  We define $I^\ast = I \sm
    \{i^\ast\}$ and $J^\ast = J \sm \{j^\ast\}$. Finally, we define
    $$
        Y^\ast = \left( U^k \times \Ac^{l-1} \right)
          \cup \left( \bigcup_{j \in J^\ast} C_{j,k+l-1} \right)
          \sm \left( \bigcup_{i \in I^\ast} C_{i,k+l-1} \right).
    $$
    By the induction hypothesis, $U^{k+l-1}$ can be partitioned into copies of
    $A$ and $B$. Moreover, $Y = \left( Y^\ast \times \Ac \right) \cup
    C_{k+l,k+l} \sm C_{i^\ast,k+l}$, so we can apply \Cref{prop:modify} to
    finish the proof in this case.

    On the other hand, if $k+l \not\in J$, then we define
    $$
        Y^\prime = \left( U^k \times \Ac^{l-1} \right)
          \cup \left( \bigcup_{j \in J} C_{j,k+l-1} \right)
          \sm \left( \bigcup_{i \in I} C_{i,k+l-1} \right)
    $$
    and observe that $Y = Y^\prime \times \Ac$. Moreover, $U^{k+l-1} \sm
    Y^\prime$ can be partitioned into copies of $A$ and $B$ by the induction
    hypothesis, and hence it follows from \Cref{prop:blowup} that the same holds
    for $U^{k+l} \sm Y$.
\end{proof}

Recall that our ultimate goal in this subsection is to partition $S^2 \times
U^n$, for some $n \ge 1$, into copies of members of $\mathcal{F} \cup \{A, B\}$.
We cannot achieve this goal just yet, but we have already provided ourselves
with tools, in the form of
\Cref{prop:blowup,prop:onecorner,prop:modify,cor:multiplechanges}, that allow us
to partition $U^k \setminus X$, for various $k$ and various sets $X$, into
copies of $A$ and $B$. Our strategy now can be roughly described as follows. We
will take a large $n$ and we will slice $S^2 \times U^n$ up into copies of $S
\times U^n$.  We will partition big parts of these slices into copies of members
of $\F \cup \{A, B\}$, leaving out gaps that we can control. Then we will
combine the gaps across all slices, and we will fill them in with copies of
members of $\mathcal{F}$. The following proposition will tell us what gaps we
should leave in the slices so that their union could be filled in later on.

\begin{prop} \label{prop:fillin}
    
    Let $t$ be a positive integer and take not necessarily distinct sets
    $P_1, \dotsc, P_t \in \F$. Define $Q_0, \dotsc, Q_t \subset S \times
    U^t$ by setting 
    $$
        Q_i = 
        \begin{cases}
            P_i \times C_{i,t} & \text{if } 1 \le i \le t, \\
            S \times C_{0,t} & \text{if } i = 0.
        \end{cases}
    $$
    Then the set $(S \times U^t) \sm (Q_0 \cup \dotsb \cup Q_t)$ can be
    partitioned into copies of members of $\F \cup \{A \cup B\}$.

\end{prop}

\begin{proof}
    We use induction on $t$. We take $t = 0$ to be the base case. Although the
    set $C_{0,0}$ had not been defined, we may interpret $S \times C_{0,0}$
    and $S \times U^0$ as both being the set $S$, in which case the
    conclusion says that the empty set can be partitioned into copies of members
    of $\F \cup \{A, B\}$, which is trivially true.

    Now suppose that $t \ge 1$. We write $X = Q_0 \cup \dotsb \cup Q_t$ and
    $X^\ast = (S \times C_{0,t-1}) \cup (P_1 \times C_{1,t-1}) \cup \dotsb
    \cup (P_{t-1} \times C_{t-1,t-1})$. By the induction hypothesis, $(S
    \times U^{t-1}) \sm X^\ast$ can be partitioned into copies of $\F \cup \{A,
    B\}$. Moreover, using the fact that $X = (X^\ast \times \Ac) \cup Q_t$, we
    can partition $(S \times U^t) \sm X$ into three sets $Y_1, Y_2, Y_3$,
    where
    \begin{align*}
        Y_1 &= \left( (S \times U^{t-1}) \sm X^\ast \right) \times \Ac, \\
        Y_2 &= \left( (S \times U^{t-1}) \sm (P_t \times C_{0,t-1}) \right)
               \times A), \\
        Y_3 &= P_t \times U^{t-1} \times (A \cap B).
    \end{align*}
    It is clear from \Cref{fig:blueprint} that these sets do partition $(S
    \times U^t) \sm X$. Moreover, $Y_1$ can be partitioned into copies of
    members of $\F \cup \{A, B\}$ (by the induction hypothesis); $Y_2$ is
    trivially a disjoint union of copies of $A$; $Y_3$ is a disjoint union of
    copies of $P_t$, which is a member of $\mathcal{F}$.
\end{proof}

\begin{figure}[ht]
    \begin{centering}
        \begin{tikzpicture}
            \draw[thick] (0,0)--(7,0)--(7,4)--(0,4)--cycle;
            \draw[thick,fill=lightgray] (0,0)--(4,0)--(4,1)--(0,1)--cycle;
            \draw[thick,fill=lightgray] (2,3)--(5,3)--(5,4)--(2,4)--cycle;
            \draw[thick,dashed] (4,1)--(7,1);
            \draw[thick,dashed] (2,1)--(2,3);
            \draw[thick,dashed] (5,1)--(5,3);
            \draw (5.5,0.5) node {$Y_1$};
            \draw (1,2.5) node {$Y_2$};
            \draw (6,2.5) node {$Y_2$};
            \draw (3.5,2) node {$Y_3$};
            \draw (2, 4.1) -- (2,4.2) -- node[above]{$P_t \times C_{0,t-1}$} (5,4.2) -- (5,4.1);
            \draw (0, -0.1) -- (0, -0.2) -- node[below]{$X^\ast$} (4, -0.2) -- (4, -0.1);
            \draw (7.1,0) -- (7.2,0) -- node[right]{$\Ac$}(7.2, 1) -- (7.1, 1);
            \draw (7.1,3) -- (7.2,3) -- node[right]{$\Bc$}(7.2, 4) -- (7.1, 4);
            \draw[<->] (0,5.3) -- node[above]{$S \times U^{t-1}$} (7,5.3);
            \draw[<->] (-0.8,0) -- node[left] {$U$} (-0.8,4);
        \end{tikzpicture}
        \caption{The set $X$ is shaded; $Y_1, Y_2, Y_3$ partition $(S \times U^{t-1}) \sm X$.}
        \label{fig:blueprint}
    \end{centering}
\end{figure}

We recall that, for some positive integer $r$, $\F$ contains an \rpart{r} of
$S$. In other words, there exist not necessarily distinct sets $P_1, \dots, P_m
\in \F$ such that every element of $S$ is contained in precisely $r$ of them. We
will use the sets $P_1, \dotsc, P_m$ to prove the following proposition.

\begin{prop} \label{prop:manychoices}

        Let $r$ be as above.  For any positive integer $k$ there exists an
        integer $l \ge k$ with the following property. For any distinct numbers
        $j_1, \dotsc, j_t \in \{1, \dotsc, l\}$, if $t \le k$ and $t \equiv 1
        \pmod{r}$, then the set $S \times \left( U^l \sm \bigcup_{u=1}^t
        C_{j_u,l} \right)$ can be partitioned into copies of members of $\F \cup
        \{A, B\}$.

\end{prop}

\begin{proof}[Proof (see \Cref{fig:manychoices}).]
    Given $k$, fix any $l \ge k + (k-1)m/r$. Given distinct $j_1, \dotsc,
    j_t \in \{1, \dotsc, l\}$, we may assume (after a permutation of
    coordinates, if necessary), that $\{j_1, \dotsc, j_t\} = \{l-t+1, \dotsc,
    l\}$. We denote this set by $J$. Since $t \le k$ and $t \equiv 1 \pmod{r}$ by assumption,
    we may write $t = ar + 1$ for some integer $0 \le a \le (k-1)/r$. We will
    prove that the set
    $$
        Y = S \times \left( U^l \sm \bigcup_{j\in J} C_{j,l} \right)
    $$
    can be partitioned into copies of members of $\F \cup \{A, B\}$.
    
    Extend $P_1, \dotsc, P_m$ to a longer list $P_1, \dotsc, P_{am}$ by setting
    $P_{i+m} = P_i$ for every $m+1\le i\le am$. The only important property of
    this new list is that every member of the original list is repeated exactly
    $a$ times. Moreover, set $P_0 = S$. Then every element of $S$ is contained
    in exactly $ar+1 = t$ members of the list $P_0, \dotsc, P_{am}$. We define
    $$
        X = (S \times U^{am} \times \Ac^{l-am}) \sm \left(
        \bigcup_{i=0}^{am} P_i \times C_{i,l} \right).
    $$
    Since $X = \left( (S \times U^{am}) \sm \left( \bigcup_{i=0}^{am} P_i \times
    C_{i,am} \right) \right) \times \Ac^{l-am}$, it follows from
    \Cref{prop:fillin} that $X$ can be partitioned into copies of members of $\F
    \cup \{A, B\}$. Since $\min J > l-k \ge am$, the set $X$ is disjoint from $S
    \times C_{j,l}$ for any $j \in I$, and hence $X \subset Y$. Therefore, it
    only remains to prove that $Y \setminus X$ can be partitioned into copies of
    members of $\F \cup \{A, B\}$.

    For any $z \in S$, we denote by $S_z$ the cross-section of $Y \sm X$ at $z$,
    that is,
    $$
        Y_z = \{ y \in U^l : (z, y) \in Y \sm X \}.
    $$
    For the moment, let us focus on one fixed $z \in S$. By construction of
    $P_0, \dotsc, P_{am}$, there are exactly $t$ values of $i$ for which $z \in
    P_i$. Let $I$ be the set of these values. Then
    $$
        Y_z = U^l \sm \left( \left( U^{am} \times \Ac^{l-am} \right)
            \cup \left( \bigcup_{j \in J} C_{j,l} \right) \sm \left(
            \bigcup_{i\in I} C_{i,l} \right)\right).
    $$
    Since $|I| = |J| = t$, $I \subset \{0, \dotsc, am\}$ and $J \subset \{am+1,
    \dotsc, l\}$, \Cref{cor:multiplechanges} implies that $Y_z$ can be
    partitioned into copies of members of $\F \cup \{A, B\}$.
    
    Now we are done: $Y = X \cup \left( \bigcup_{z \in S} \{z\} \times
    Y_z \right)$, and we have proved that $X$ and every $Y_z$ can be
    partitioned into copies of members of $\F \cup \{A, B\}$.
\end{proof}

\begin{figure}[ht]
    \begin{centering}
        \begin{tikzpicture}
            \draw (2.40,7) node{$\cdots$};
            \draw (5.20,7) node{$\cdots$};
            \draw (10.50,7) node{$\cdots$};
            \fill[pattern=north east lines] (0.00,2.40)--(12.80,2.40)--(12.80,3.20)--(0.00,3.20)--cycle;
            \draw[lightgray,fill=lightgray] (0.00,0.00)--(7.20,0.00)--(7.20,6.40)--(0.00,6.40)--cycle;
            \draw[fill=white] (0.08,0.00)--(0.72,0.00)--(0.72,6.40)--(0.08,6.40)--cycle;
            \draw (0.08,6.50)--(0.08,6.70)--node[above]{$C_{0,l}$}(0.72,6.70)--(0.72,6.50);
            \draw[fill=white] (0.88,0.00)--(1.52,0.00)--(1.52,3.20)--(0.88,3.20)--cycle;
            \draw (0.88,6.50)--(0.88,6.70)--node[above]{$C_{1,l}$}(1.52,6.70)--(1.52,6.50);
            \draw[fill=white] (1.68,1.60)--(2.32,1.60)--(2.32,4.80)--(1.68,4.80)--cycle;
            \draw[fill=white] (2.48,3.20)--(3.12,3.20)--(3.12,6.40)--(2.48,6.40)--cycle;
            \draw[fill=white] (3.28,4.80)--(3.92,4.80)--(3.92,6.40)--(3.28,6.40)--cycle;
            \draw (3.28,6.50)--(3.28,6.70)--node[above]{$C_{m,l}$}(3.92,6.70)--(3.92,6.50);
            \draw[fill=white] (3.28,0.00)--(3.92,0.00)--(3.92,1.60)--(3.28,1.60)--cycle;
            \draw[fill=white] (4.08,0.00)--(4.72,0.00)--(4.72,3.20)--(4.08,3.20)--cycle;
            \draw[fill=white] (4.88,1.60)--(5.52,1.60)--(5.52,4.80)--(4.88,4.80)--cycle;
            \draw[fill=white] (5.68,3.20)--(6.32,3.20)--(6.32,6.40)--(5.68,6.40)--cycle;
            \draw[fill=white] (6.48,4.80)--(7.12,4.80)--(7.12,6.40)--(6.48,6.40)--cycle;
            \draw[fill=white] (6.48,0.00)--(7.12,0.00)--(7.12,1.60)--(6.48,1.60)--cycle;
            \draw (6.48,6.50)--(6.48,6.70)--node[above]{$C_{am,l}$}(7.12,6.70)--(7.12,6.50);
            \draw[fill=white] (12.08,0.00)--(12.72,0.00)--(12.72,6.40)--(12.08,6.40)--cycle;
            \draw (12.08,6.50)--(12.08,6.70)--node[above]{$C_{l,l}$}(12.72,6.70)--(12.72,6.50);
            \draw[fill=white] (11.28,0.00)--(11.92,0.00)--(11.92,6.40)--(11.28,6.40)--cycle;
            \draw (11.28,6.50)--(11.28,6.70)--node[above]{$C_{l-1,l}$}(11.92,6.70)--(11.92,6.50);
            \draw[fill=white] (10.48,0.00)--(11.12,0.00)--(11.12,6.40)--(10.48,6.40)--cycle;
            \draw[fill=white] (9.68,0.00)--(10.32,0.00)--(10.32,6.40)--(9.68,6.40)--cycle;
            \draw[fill=white] (8.88,0.00)--(9.52,0.00)--(9.52,6.40)--(8.88,6.40)--cycle;
            \draw (8.88,6.50)--(8.88,6.70)--node[above]{$C_{l-t+1,l}$}(9.52,6.70)--(9.52,6.50);
            \draw[thick,fill=none] (0.00,0.00)--(12.80,0.00)--(12.80,6.40)--(0.00,6.40)--cycle;
            \draw[<->] (-0.80,0.00)--node[left]{$S$}(-0.80,6.40);
            \draw[<->] (0.00,-0.80)--node[below]{$U^l$}(12.80,-0.80);
            \draw[thick,dashed] (0.02,2.40)--(12.78,2.40)--(12.78,3.20)--(0.02,3.20)--cycle;
            \draw (13.20,2.80) node{$Y_z$};
            \draw[pattern = north east lines] (0.08,2.40)--(0.72,2.40)--(0.72,3.20)--(0.08,3.20)--cycle;
            \draw[pattern = north east lines] (0.88,2.40)--(1.52,2.40)--(1.52,3.20)--(0.88,3.20)--cycle;
            \draw[pattern = north east lines] (1.68,2.40)--(2.32,2.40)--(2.32,3.20)--(1.68,3.20)--cycle;
            \draw[pattern = north east lines] (4.08,2.40)--(4.72,2.40)--(4.72,3.20)--(4.08,3.20)--cycle;
            \draw[pattern = north east lines] (4.88,2.40)--(5.52,2.40)--(5.52,3.20)--(4.88,3.20)--cycle;
        \end{tikzpicture}
        \caption{The set $X$ is shaded, a slice $Y_z$ is hatched diagonally.
            \Cref{prop:fillin,cor:multiplechanges}, respectively, imply that $X$
            and $Y_z$ can be partitioned into copies of members of $\F \cup \{A,
            B\}$.}
        \label{fig:manychoices}
    \end{centering}
\end{figure}
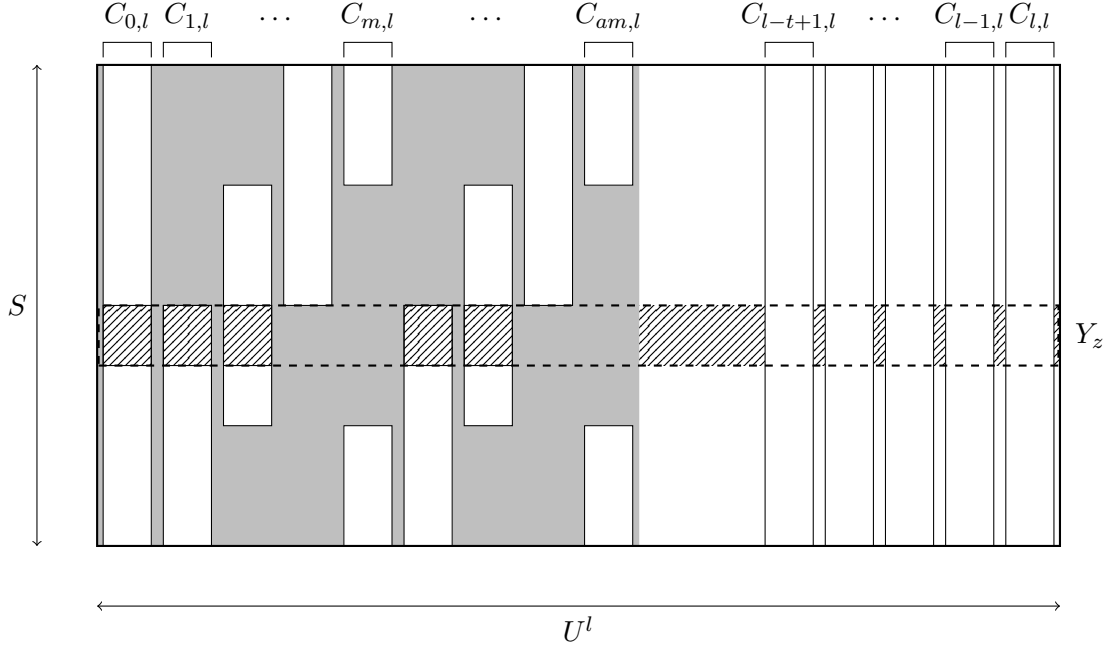

We are now ready to prove \Cref{lem:main}.

\begin{proof}[Proof of \Cref{lem:main}.]
    We begin by recalling that $\F$ contains a \modpart{r} of $S$. In other
    words, there exists a family of not necessarily distinct sets $R_1, \dotsc,
    R_k \in \F$ such that every $x \in S$ is contained in exactly $1 + ra_x$
    members of this family, where $a_x$ is an integer. Furthermore,
    \Cref{prop:manychoices} provides us with a positive integer $n \ge k$ such
    that, for any set $I \subset \{1, \dotsc n\}$ that satisfies $|I| \equiv 1
    \pmod{r}$ and $|I| \le k$, the set $S \times \left( U^n \sm \bigcup_{i \in
    I} C_{i,n} \right)$ can be partitioned into copies of members of $\F \cup
    \{A, B\}$. We will show that $S \times U^n$ can be partitioned into copies
    of members of $\F \cup \{A, B\}$.

    We define
    $$
        X = \left( S \times U^n \right)
            \sm \left( \bigcup_{i=1}^k R_k \times C_{i,n}  \right)
    $$
    and, for any $y \in S$, we let $X_y$ denote the cross-section of $X$ at $y$,
    that is, $X_y = \{ x \in U^n : (y, x) \in X \}$. Any $y \in S$ is contained
    in $1 + ra_y$ members of the family $R_1, \dotsc, R_k$. Therefore, if we
    write $J_y = \{ j \in [k] : y \in R_j\}$, then $|J_y| \equiv 1
    \pmod{r}$ and $|J_y| \le k$. Moreover, it is easy to see that
    $$
        X_y = U^n \sm \left( \bigcup_{j \in J_y} C_{j,n} \right).
    $$
    By \Cref{prop:manychoices}, $S \times X_y$ can be partitioned into copies of
    members of $\F \cup \{A, B\}$. Therefore, so can be $S \times X$, which is
    the disjoint union of sets $S \times \{y\} \times X_y$, $y \in S$.

    Finally, observe that $S^2 \times U^n$ is the disjoint union of $S \times X$
    and sets $S \times R_i \times C_{i,n}$, $1 \le i \le k$. Each set $S \times
    R_i \times C_{i,n}$ is trivially a union of disjoint copies of $R_i$, which
    is a member of $\F$. Therefore, $S^2 \times U^n$ can be partitioned into
    copies of members of $\F \cup \{A, B\}$, as required.
\end{proof}

\section{Weak partitions} \label{sec:particular}

\subsection{Constructing an \rpart{r} of $\BL{n}$}

Our aim in this subsection is to prove \Cref{lem:r-partition}, which asserts the
existence of an $r$-partition of $\BL{n}$ into copies of $P$ for some $n, r$.
Our proof is somewhat technical, but not very difficult.

Recall that by our earlier definition a weight function is an assignment of 
non-negative integer weights to sets from some selected family. We now
extend this definition to allow more general weights. Namely, given a
set $V \subset \R$ and a set family $\F$, a \emph{$V$-valued weight function on
$\F$} is a function $w : \F \to V$. Usually, we will take $V$ to be $\Z$, $\Z^+$
or $\Q^+$, where $S^+$ is defined to be $S \cap [0, \infty)$ for any $S \subset
\R$. We note that a weight function in the old sense is precisely a
$\Z^+$-valued weight function in the new sense.

Moreover, if $\F$ is a family of subsets of some set $X$, for any $x \in X$ we
define the \emph{multiplicity of $x$ for $w$}, denoted $N_w(x)$, to be the
total weight assigned to the members of $\F$ that contain $x$. That is,
$$
N_w(x) = \sum_{\substack{A \in \F \\ x \in A}} w(A).
$$
Moreover, for any $Y \subset X$, we set $N_w(Y) = \sum_{y \in Y} N_w(y)$.  With
these definition at hand, we can restate \Cref{lem:r-partition} in a form that
is slightly more convenient for the proof.

\begin{lemprimed}{lem:r-partition} \label{lem:r-partition-primed}

    Let $P$ be a finite poset with a greatest and a least element. Then
    there exist a positive integer $n$ and a $\Q^+$-valued weight function
    $w$ on the copies of $P$ in $\BL{n}$ such that $N_w(x) = 1$ for all $x \in
    \BL{n}$.

\end{lemprimed}

To see why \Cref{lem:r-partition-primed} is equivalent to
\Cref{lem:r-partition}, observe that a $\Q^+$-valued weight function $w$ on a
finite set family $\F$ can be made into a $\Z^+$-valued weight function by
multiplying it by the least common multiple of the denominators of the $w(A)$
for $A \in \F$.  Moreover, if $N_w(x) = 1$ for all $x$, then the resulting
$\Z^+$-valued weight function $rw$ satisfies $N_{rw}(x) = r$ for all $x$.

The main idea in the proof is to look for a weight function that is symmetric
with respect to all permutions of the ground set $\{1, \dotsc, n\}$. Such a
weight function can be obtained by averaging any another weight function over
all permutations of $\{1, \dotsc, n\}$. This idea essentially removes the need
to consider the structure of the poset $P$, and converts
\Cref{lem:r-partition-primed} into a question about finding a certain weight
function on the power set of $\{0, \dotsc, n\}$. This is reflected in the
following definition.

Let $P$ be a poset and $n$ a positive integer. Moreover, let $w$ be a
$\Q^+$-valued weight function on the copies of $P$ in $\BL{n}$. We define a new
$\Q^+$-valued weight function $\sym{w}$, also on the copies of $P$ in $\BL{n}$,
by setting
$$
\sym{w}(A) = \frac{1}{n!} \sum_{\pi \in \textrm{Perm}(n)} w\big( \pi(A) \big)
$$
for all $A$ that are copies of $P$ in $\BL{n}$. Here $\textrm{Perm}(n)$ denotes
the set of permutations of $\{1, \dotsc, n\}$ and we recall that $\pi(A)$
denotes the image of $A$ after permuting the coordinates of $\BL{n}$ according
to $\pi$.

Since elements of $\BL{n}$ are subsets of $\{1, \dotsc, n\}$, it makes sense to
write $|x|$ for $x \in \BL{n}$ to denote the size of $x$. We partition $\BL{n}$
into levels $L_0, \dotsc, L_n$, where $L_k = \{ x \in \BL{n} : |x| = k\}$. Then,
for any $x \in L_k$,
$$
N_{\sym{w}}(x) = \frac{1}{\binom{n}{k}} N_w(L_k).
$$
Therefore, our task is reduced to finding $w$ such that $N_w(L_k) =
\binom{n}{k}$ for all $k$. To this aim, we would like to have a tool for
embedding $P$ into $\BL{n}$ while keeping control on levels into which we
map the elements of $P$. The following proposition provides us with such a tool.

We say that a set $A \subset \Z$ is \emph{$d$-scattered} if, for any distinct
$i, j \in A$, we have $|i - j| \ge d$.

\begin{prop} \label{prop:embeddings}
    Let $P$ be a finite poset with a greatest and a least element. Then there
    exists a positive integer $d$ such that, for any integer $n \ge (|P|-1)d$
    and any $d$-scattered set $A \subset \{0, \dotsc, n\}$ of size $|P|$, there
    exists an embedding $\phi : P \to \BL{n}$ satisfying
    $$
        \{|\phi(x)| : x \in P\} = A.
    $$
    In other words, for any $0 \le k \le n$,
    $$
        |L_k \cap \phi(P)| =
        \begin{cases*}
            1 & if $k \in A$, \\
            0 & otherwise.
        \end{cases*}
    $$
\end{prop}

\begin{proof}
    We start by recalling that, since $P$ is finite, it can be embedded into
    $\BL{k}$ for some $k$. Let $\psi : P \to \BL{k}$ be an embedding which maps
    the greatest element of $P$ to the greatest element of $\BL{k}$ and the
    least element of $P$ to the least element of $\BL{k}$. We write $s =
    |P|$ and list the elements of $P$ as $p_1, \dotsc, p_s$ in the order where
    $0 = |\psi(p_1)| \le \dotsb \le |\psi(p_s)| = k$.

    We will prove that $d = k$ works. Indeed, take any integer $n \ge (s-1)k$
    and let $A \subset \{0, \dotsc, n\}$ be a $k$-scattered set of size $s$.
    Then $A = \{a_1, \dotsc, a_s\}$, where $0 \le a_1 < \dotsb < a_s \le n$ and
    $a_{i+1} \ge a_i + k$ for all $0 \le i \le s-1$. For every $1 \le i \le s$,
    we set
    $$
        \phi(p_i) = \psi(p_i) \cup \{k+1, \dotsc, k + a_i - |\psi(p_i)|\}.
    $$
    To prove that $\phi : P \to \BL{n}$ is a well-defined embedding, we have to
    check that $0 \le a_1 - |\psi(p_1)| \le \dotsb \le a_s - |\psi(p_s)| \le
    n-k$.  However, if we prove this, then it is trivial to see that
    $|\phi(p_i)| = a_i$ for all $i$, as required.

    First, we observe that $a_1 - |\psi(p_1)| = a_1 \ge 0$ and $a_s -
    |\psi(p_s)| = a_s - k \le n - k$. Furthermore, for any $1 \le i \le s-1$, we
    have $a_{i+1} - |\psi(p_{i+1})| \ge a_i + k - k = a_i \ge a_i -
    |\psi(p_i)|$, and so we are done.
\end{proof}

\begin{prop} \label{prop:inpowerset}
    
    Let $X$ be a finite set and $t$ a positive integer. If $f : X \to \Q^+$ is a
    function such that
    $$
        t \max_{x \in X} f(x) \le \sum_{x \in X} f(x),
    $$
    then there exists a $\Q^+$-valued weight function $w$ on the family of
    $t$-element subsets of $X$, such that $N_w(x) = f(x)$ for all $x \in X$.

\end{prop}

\begin{proof}
    Let $r$ be the least common multiple of the denominators of the $f(x)$ over
    all $x \in X$.  After multiplying $f$ by $tr$, we may assume that $f$ takes
    values in $\Z^+$ and that $\sum_{x \in X} f(x)$ is divisible by $t$. We
    denote $\sum_{x \in X} f(x) = Nt$ and we will use induction on $N$.

    If $f(x) = 0$ for all $x \in X$, then the result is trivial. Therefore, we
    may assume that $N \ge 1$. Let
    $S = \{x \in X : f(x) > 0\}$ and $T = \{x \in X : f(x) = N\}$. Since
    $$
        t \max_{x \in X} f(x) \,\le\, \sum_{x \in X} f(x) \,\le\, 
        |S| \max_{x \in X} f(x),
    $$
    it follows that $|S| \ge t$. Moreover, $N|T| \le \sum_{x \in X} f(x) = Nt$,
    and hence $|T| \le t$. Therefore, there exists a set $A$ such that $T
    \subset A \subset S$ and $|A| = t$.

    We define $g : X \to \Z^+$ by setting
    $$
        g(x) =
        \begin{cases*}
            f(x) - 1 & if $x \in A$, \\
            f(x) & otherwise.
        \end{cases*}.
    $$
    Then $\sum_{x \in X} g(x) = (N-1)t$ is non-negative and divisible by $t$.
    Moreover, since $T \subset A$, we have $g(x) \le N-1$ for all $x \in X$.
    Therefore, by the induction hypothesis, there exists a $\Q^+$-valued weight
    function $w'$ on the $t$-element subsets of $X$, such that $N_{w'}(x) =
    g(x)$ for all $x \in X$. We define
    $$
        w(B) =
        \begin{cases*}
            w'(A) + 1 & if $B = A$, \\
            w'(B) & if $B \subset X$, $|B| = t$ and $B \neq A$.
        \end{cases*}
    $$
    This $w$ satisfies the required conditions.
\end{proof}

It is easy to deduce \Cref{lem:r-partition-primed} from
\Cref{prop:embeddings,prop:inpowerset}.

\begin{proof}[Proof of \Cref{lem:r-partition-primed}.]
    Let $P$ be a finite poset with a greatest and a least element. Recall
    that our aim is to find, for some positive integer $n$, a $\Q^+$-valued
    weight function $w$ on the copies of $P$ in $\BL{n}$, such that $N_w(L_i) =
    \binom{n}{i}$ for all $0 \le i \le n$. Indeed, then $N_{\sym{w}}(x) = 1$ for
    all $x \in \BL{n}$.

    Let $d$ be such that, for any $n \ge (|P|-1)d$ and any $d$-scattered set $A
    \subset \{0, \dotsc, n\}$ of size $|P|$, there exists a copy of $P$ in
    $\BL{n}$, say $C$, such that $\{|x| : x \in C\} = A$. The existence of such
    a number $d$ is guaranteed by \Cref{prop:embeddings}. Set $k = |P|d$.

    Choose $n$ large enough to satisfy the inequality $k \binom{n}{\lceil n/2
    \rceil} \le 2^n$. Then \Cref{prop:inpowerset} gives a $\Q^+$-valued weight
    function $w'$ on the $k$-element subsets of $\{0, \dotsc, n\}$ that satisfies
    $N_{w'}(i) = \binom{n}{i}$ for all $0 \le i \le n$.
    
    Let $B$ be a $k$-element subset of $\{0, \dotsc, n\}$. If we consider the
    elements of $B$ in increasing order and take every $d$th element, we obtain
    a $d$-scattered set. In this way we can partition $B$ into $d$-scattered
    sets $B_1, \dotsc, B_d$, each of size $k/d = |P|$. We say that $B$
    \emph{splits} into sets $B_1, \dotsc, B_d$.

    By splitting $k$-element sets we obtain a $\Q^+$-valued weight function
    $w''$ on $d$-scattered $|P|$-element subsets of $\{0, \dotsc, n\}$.  More
    precisely, we define $w''(A) = \sum w'(B)$, summing over all $k$-element
    sets $B \subset \{0, \dotsc, n\}$ with the property that $A$ is one of the
    sets into which $B$ splits. Note that we have $N_{w''}(i) = N_{w'}(i) =
    \binom{n}{i}$ for all $0 \le i \le n$.

    Finally, for any $d$-scattered $|P|$-element set $A \subset \{0, \dotsc,
    n\}$ we choose one copy of $P$ in $\BL{n}$, denoted $C_A$, such that $\{|x|
    : x \in C_A\} = A$. We define a $\Q^+$-valued weight function $w$ on the
    copies of $P$ in $\BL{n}$ by setting
    $$
        w(C) =
        \begin{cases*}
            w''(A) & if $C = C_A$ for some $d$-scattered $|P|$-element set $A
                     \subset \{0, \dotsc, n\}$, \\
            0 & otherwise.
        \end{cases*}
    $$
    We note that every $d$-scattered $|P|$-element set $A \subset \{0, \dotsc,
    n\}$ contributes $w''(A)$ towards both $N_{w''}(i)$ and $N_w(L_i)$ for every
    $i \in A$, and $0$ towards both $N_{w''}(j)$ and $N_w(L_j)$ for every $j
    \not\in A$.  Therefore, $N_w(L_i) = N_{w''}(i) = \binom{n}{i}$ for all $0
    \le i \le n$, as required.
\end{proof}

\subsection{Constructing a \modpart{r} of $\BL{n}$}

Here we prove \Cref{lem:1-mod-r-partition}, which asserts the existence of an
\modpart{r} of $\BL{n}$ into copies of $P$ for some $n$. This proof is shorter,
but slightly trickier than that of \Cref{lem:r-partition}. We begin by recasting
\Cref{lem:1-mod-r-partition} in a form which is stronger, but more convenient to
work with.

\begin{lemprimed}{lem:1-mod-r-partition} \label{lem:1-mod-r-partition-primed}

    Let $P$ be a poset of size $2^k$ with a greatest and a least element.  Then
    there exist a positive integer $n$ and a $\Z$-valued weight function $w$ on
    the copies of $P$ in $\BL{n}$ satisfying $N_w(x) = 1$ for all $x \in
    \BL{n}$.

\end{lemprimed}

We remark that \Cref{lem:1-mod-r-partition-primed} does imply
\Cref{lem:1-mod-r-partition}, because the $\Z$-valued weight function $w$ can be
converted into a suitable $\Z^+$-valued weight function $w'$ by choosing $w'(A)
\in \{0, \dotsc, r-1\}$ such that $w'(A) \equiv w(A) \pmod{r}$, for all $A$.

\begin{proof}[Proof of \Cref{lem:1-mod-r-partition-primed}.]
    Since $P$ is finite, it can be embedded into $\BL{d}$, for some $d$, by an
    embedding which maps the greatest and the least elements of $P$ to the
    corresponding elements of $\BL{d}$. We will show that $n = 2d-1$ works.

    We say that a function $f : \BL{n} \to \Z$ is \emph{realisable} if there
    exists a $\Z$-valued weight function $w$ on the copies of $P$ in $\BL{n}$,
    such that $N_w(x) = f(x)$ for all $x \in \BL{n}$. We note that
    if $f, g$ are realisable functions, then so are $f+g$ and $f-g$. Our aim is
    to show that the constant $1$ function on $\BL{n}$ is realisable.

    For any $A \subset \BL{n}$, we define $\ind{A} : \BL{n} \to \{0,1\}$ to be
    the indicator function of $A$. Clearly, if $A$ is a copy of $P$, then
    $\ind{A}$ is realisable.

    We denote the greatest and the least elements of $\BL{n}$ by $x_+, x_-$. Let
    $x \in \BL{n}$. If $|x| \ge d$, then there exists an embedding $\BL{d} \to
    \BL{n}$ which maps the greatest element of $\BL{d}$ to $x$. Therefore, in
    $\BL{n}$, we can find a copy of $P$ whose greatest element is $x$. We denote
    this copy by $A$. Moreover, if we denote $B = A \sm \{x\}$, then $B \cup
    \{x\}$ and $B \cup \{x_+\}$ are copies of $P$. Therefore, the function
    $\ind{\{x\}} - \ind{\{x_+\}} = \ind{B\cup\{x\}} - \ind{B\cup\{x_+\}}$ is
    realisable.

    Similarly, if $|x| \le d$, then there exists an embedding $\BL{d} \to
    \BL{n}$ which maps the least element of $\BL{d}$ to $x$. Then we can find a
    copy of $P$ in $\BL{n}$, which we denote by $A$, with the property that $x$
    is the least element of $A$. We write $B = A \sm \{x\}$ and observe that $A
    \cup \{x\}$ and $A \cup \{x_-\}$ are copies of $P$. Therefore, the function
    $\ind{\{x\}} - \ind{\{x_-\}} = \ind{B\cup\{x\}} - \ind{B\cup\{x_-\}}$ is
    realisable.

    In particular, for any $x \in \BL{n}$, at least one of the functions
    $\ind{\{x\}} - \ind{\{x_+\}}$ and $\ind{\{x\}} - \ind{\{x_-\}}$ is
    realisable.  Moreover, if $|x| = d$, then both of them are. Therefore, by
    choosing any $x_0 \in \BL{n}$ with $|x_0| = d$, we can see that
    $\ind{\{x_+\}} - \ind{\{x_-\}} = (\ind{\{x_0\}} - \ind{\{x_-\}}) -
    (\ind{\{x_0\}} - \ind{\{x_+\}})$ is realisable. We conclude that, in fact,
    for any $x,y \in \BL{n}$, the function $\ind{\{x\}} - \ind{\{y\}}$ is
    realisable.

    Let $f, g : \BL{n} \to \Z$ be two functions that satisfy $\sum_{x\in\BL{n}}
    f(x) = \sum_{x\in\BL{n}}g(x)$. Then the difference $f-g$ can be expressed as
    a sum of functions of the form $\ind{\{x\}} - \ind{\{y\}}$ with $x, y \in
    \BL{n}$, so $f-g$ is realisable. Hence, $f$ is realisable if and only if $g$
    is realisable. Therefore, to prove that the constant $1$ function is
    realisable, it is enough to find one realisable function $f$ such that
    $\sum_{x\in\BL{n}} f(x) = 2^n$. However, we know that $|P|=2^k$ and, trivially,
    $k \le n$, so we can take $f = 2^{n-k}\cdot\ind{A}$ for any $A \subset \BL{n}$
    which is a copy of $P$.
\end{proof}

\section{Concluding remarks and open problems}
\label{sec:final}

    In the proof of \Cref{thm:main} we do not explicitly keep track of a value
    of $n$ that would be sufficient. This is to make the proof more readable.
    Moreover, we did not put any serious effort into finding a good bound. The
    following bound can be extracted from the proof.

    \begin{thmprimed}{thm:main}
        There exists an absolute constant $C>0$ with the following property. Let
        $P$ be a poset of size $2^k$ with a greatest and a least element. Then,
        for any integer $n \ge 2^{|P|^C}$, the Boolean lattice $\BL{n}$ can be
        partitioned into copies of $P$.
    \end{thmprimed}

    It is interesting to ask what happens if $P$ does not satisfy the conditions
    required by \Cref{thm:main}. Of course, then it is impossible to partition
    $\BL{n}$ into copies of $P$. However, what if we are allowed to leave a small
    number of elements of $\BL{n}$ uncovered? For example, if $P$ does not have
    a greatest and/or a least element, then the greatest and/or the least
    element of $\BL{n}$ are the only ones that obviously cannot be covered by
    copies of $P$. Lonc~\cite{Lonc1991} conjectured that, if $n$ is large and if
    an obvious divisibility condition is satisfied, then $\BL{n}$ with its
    greatest and least element removed can be partitioned into copies of $P$.

    \begin{conj}[Lonc] \label{conj:loncs-open}
        Let $P$ be a finite poset. If $n$ is sufficiently large and if $|P|$
        divides $2^n - 2$, then it is possible to partition $\BL{n}$, with its
        greatest and least element removed, into copies of $P$.
    \end{conj}
    
    In the spirit of Griggs' conjecture it is reasonable to hope that, even if
    we do not impose any divisibility conditions for $|P|$, for sufficiently
    large $n$, $\BL{n}$ can be partitioned into copies of $P$ and a set of size
    $c$, where $c < |P|$. Or perhaps one can bound $c$ by a weaker constant
    which depends on $P$. 

    \begin{restatable}{qn}{open} \label{conj:open}
        Let $P$ be a finite poset. Must there exist a constant $c = c(P)$ such
        that, for any $n$, it is possible to cover all but at most $c$ elements
        of $\BL{n}$ by disjoint copies of $P$?
    \end{restatable}

    We remark that \Cref{conj:loncs-open} would give a positive answer to
    \Cref{conj:open} in the case where $|P|$ is not a multiple of $4$.

\renewcommand{\biblistfont}{\normalfont\normalsize}
\bibliography{poset}

\begin{bibdiv}
\begin{biblist}

\bib{Elzobi2003}{article}{
      author={Elzobi, M.},
      author={Lonc, Z.},
       title={On partition of a boolean lattice into chains of equal sizes},
        date={2003},
     journal={J. Combin. Math. Combin. Comput.},
      volume={44},
       pages={237\ndash 249},
}

\bib{Furedi1985}{inproceedings}{
      author={F\"{u}redi, Z.},
       title={Problem session},
        date={1985},
   booktitle={Kombinatorik {G}eordneter {M}engen},
     address={Oberwolfach, FRG},
}

\bib{Griggs1988}{article}{
      author={Griggs, J.~R.},
       title={Problems on chain partitions},
        date={1988},
        ISSN={0012-365X},
     journal={Discrete Mathematics},
      volume={72},
      number={1–3},
       pages={157\ndash 162},
  url={http://www.sciencedirect.com/science/article/pii/0012365X8890204X},
}

\bib{Griggs1987}{article}{
      author={Griggs, J.~R.},
      author={Yeh, R. K.~C.},
      author={Grinstead, C.~M.},
       title={Partitioning boolean lattices into chains of subsets},
        date={1987},
        ISSN={1572-9273},
     journal={Order},
      volume={4},
      number={1},
       pages={65\ndash 67},
         url={http://dx.doi.org/10.1007/BF00346654},
}

\bib{Gruslys2016}{article}{
      author={Gruslys, V.},
      author={Leader, I.},
      author={Tan, T.~S.},
       title={Tiling with arbitrary tiles},
        date={2016},
     journal={Proc. Lond. Math. Soc.},
      volume={112},
      number={6},
       pages={1019\ndash 1039},
}

\bib{Hsu2002}{article}{
      author={Hsu, T.},
      author={Logan, M.~J.},
      author={Shahriari, S.},
      author={Towse, C.},
       title={Partitioning the boolean lattice into chains of large minimum
  size},
        date={2002},
        ISSN={0097-3165},
     journal={J. Combin. Theory Ser. A},
      volume={97},
      number={1},
       pages={62\ndash 84},
         url={http://dx.doi.org/10.1006/jcta.2001.3197},
}

\bib{Hsu2003}{article}{
      author={Hsu, T.},
      author={Logan, M.~J.},
      author={Shahriari, S.},
      author={Towse, C.},
       title={Partitioning the boolean lattice into a minimal number of chains
  of relatively uniform size},
        date={2003},
        ISSN={0195-6698},
     journal={European J. Combin.},
      volume={24},
      number={2},
       pages={219\ndash 228},
         url={http://dx.doi.org/10.1016/S0195-6698(02)00133-6},
}

\bib{Lonc1991}{article}{
      author={Lonc, Z.},
       title={Proof of a conjecture on partitions of a boolean lattice},
        date={1991},
        ISSN={1572-9273},
     journal={Order},
      volume={8},
      number={1},
       pages={17\ndash 27},
         url={http://dx.doi.org/10.1007/BF00385810},
}

\bib{Sands1985}{inproceedings}{
      author={Sands, B.},
       title={Problem session},
        date={1985},
   booktitle={Colloquium on ordered sets, {S}zeged, {H}ungary},
     address={Szeged, Hungary},
}

\bib{Tomon2015}{article}{
      author={Tomon, I.},
       title={On a conjecture of {F}\"{u}redi},
        date={2015},
        ISSN={0195-6698},
     journal={European J. Combin.},
      volume={49},
      number={C},
       pages={1\ndash 12},
         url={http://dx.doi.org/10.1016/j.ejc.2015.02.026},
}

\bib{Tomon2016}{article}{
      author={Tomon, I.},
       title={Improved bounds on the partitioning of the boolean lattice into
  chains of equal size},
        date={2016},
        ISSN={0012-365X},
     journal={Discrete Math.},
      volume={339},
      number={1},
       pages={333\ndash 343},
         url={http://dx.doi.org/10.1016/j.disc.2015.08.025},
}

\end{biblist}
\end{bibdiv}

\end{document}